\documentclass[reqno]{amsart}
\usepackage{amsmath}
\usepackage{amssymb}
  \usepackage{paralist}
  \usepackage{graphics} 
 \usepackage[colorlinks=true]{hyperref}

\usepackage{mathrsfs}
\usepackage{extarrows}
\usepackage{enumerate}

\newtheorem{theorem}{Theorem}[section]
\newtheorem{corollary}[theorem]{Corollary}

\newtheorem{lemma}[theorem]{Lemma}
\newtheorem{proposition}[theorem]{Proposition}

\theoremstyle{definition}
\newtheorem{definition}[theorem]{Definition}

\newtheorem{example}[theorem]{Example}
\newcommand{\ep}{\varepsilon}

\newcommand{\RR}{\mathbb{R}}

\newcommand{\ZZ}{\mathbb{Z}}
\newcommand{\ZZO}{\mathbb{Z}\backslash\{0\}}

\newcommand{\sC}{\mathscr{C}}

\newcommand{\sg}{\mathscr{G}}

\title[Minimality and Gluing Orbit Property]
      {Minimality and Gluing Orbit Property}

\author[Peng Sun]{}


 \email{sunpeng@cufe.edu.cn}

\begin{document}

\maketitle\ 

\centerline{\scshape Peng Sun}
\medskip
{\footnotesize
 \centerline{China Economics and Management Academy}
   \centerline{Central University of Finance and Economics}
   \centerline{Beijing 100081, China}
} 

\bigskip

\begin{abstract}
We show that a  dynamical system with gluing orbit property is either minimal or have positive topological entropy.
Moreover, for equicontinuous systems, we show that topological transitivity,
minimality and orbit gluing property are equivalent. These facts reflect the similarity and dissimilarity of gluing orbit property with 
specification like properties.
\end{abstract}


\section{Introduction}

The notion of gluing orbit property was introduced in \cite{ST}, \cite{CT}
and \cite{BV}.
As a weaker form of the well-studied specification properties, it turns out to be a more general property which still captures crucial topological features
of the systems, especially the non-hyperbolic ones. A number of results have
been obtained based on this property. See also \cite{BTV},
\cite{CLT}, \cite{TWW} and \cite{XZ}. For classical results with specification
property and specification like properties, the readers are referred to \cite{DGS}
and \cite{KLO}.

There is a remarkable difference between gluing orbit property and specification
property, as well as most weaker forms of the latter. As illustrated
in \cite{BTV} and \cite{BV}, certain examples far from specification, 
such as irrational rotations,
have gluing orbit property. We can see that gluing orbit property only requires topological
transitivity, and
is compatible with zero topological entropy,
while specification property implies topological mixing and positive topological
entropy. In general, topological mixing should not be expected for a system
that only has
gluing orbit property. For example, the direct product of the irrational
rotation and any system with specification property has  gluing orbit
property and is not topologically mixing.
In this article, we consider the entropy and  find that there is a dichotomy:
a system with gluing orbit property is either minimal or of positive topological
entropy. 

\begin{theorem}\label{posent}
Assume that $(X,f)$ is not minimal and has the gluing orbit property, then
it has positive topological entropy.
\end{theorem}

We remark that Theorem \ref{posent} is not trivial as 
there are plenty of systems with zero topological entropy
that are not minimal.
Besides those simple examples, 
there are also complicated ones (cf. Example \ref{exintmap}). We are not
clear whether there exists a system with gluing orbit property that
is both minimal and of positive topological entropy. We suspect that the
answer is positive and Herman's example \cite{Herman} 
(or something like it) 
may be a possible candidate.

As a direct corollary of Theorem \ref{posent}, 
periodic gluing orbit property implies positive
topological entropy, just as the specification properties do. The only exception
that should be ruled out
is the trivial case that $X$ consists of a single periodic orbit.

\begin{corollary}
A non-trivial system with the periodic
gluing orbit property must have positive topological entropy and exponential
growth of periodic orbits.
\end{corollary}

By Theorem \ref{posent}, for a system with zero topological
entropy, gluing orbit property implies minimality. Example \ref{exaper} shows that
the converse is not true. 
We can show that the converse holds 
if the the system is equicontinuous, in which case
gluing orbit property is also equivalent to
topological transitivity. This extends the examples in \cite{BTV}. 
We doubt if there are 
systems that are not equicontinuous,
of zero topological entropy and have gluing orbit property (hence minimal).

\begin{theorem}\label{equieq}
Assume that $(X,f)$ is equicontinuous. Then the followings
are equivalent:
\begin{enumerate}
\item $(X,f)$ is topologically transitive.
\item $(X,f)$ is minimal.
\item $(X,f)$ has gluing orbit property.
\end{enumerate}
\end{theorem}

Our results hold for both invertible and non-invertible cases, and both discrete-time
and continuous-time
cases as well. In this article we mainly work with homeomorphisms. There
are some extra technical difficulties
in the proof of the non-invertible and continuous-time cases. We give
a proof of Theorem
\ref{posent} in the semiflow case in Section 7 to illustrate the difference.

Some preliminaries are introduced in Section 2, including definitions and
notations we shall use. We prove Theorem \ref{posent} in Section 3 and discuss
some corollaries in Section 4.
Theorem \ref{equieq} is proved in Section 5. 
Some examples are investigated in Section 6.

\section{Preliminaries}
Let $(X,d)$ be a compact metric space. Let $f:X\to X$ be a homeomorphism
on $X$. Conventionally, $(X,f)$ is called a \emph{topological dynamical system} or just a system.

\begin{definition}
        $(X,f)$ is said to be \emph{equicontinuous} if for every $\ep>0$, there is $\delta>0$ such that for any $x,y\in X$ with $d(x,y)<\delta$, we have
        $$d(f^n(x),f^n(y))<\ep\text{ for every $n\ge 0$}.$$
\end{definition}

\begin{definition}
$(X,f)$ is said to be \emph{topologically transitive} if for any open sets $U,V$
in $X$, there is $n\in\ZZ$ such that
$$U\cap f^{-n}(V)\ne\emptyset.$$
\end{definition}

\begin{definition}
$(X,f)$ is said to be \emph{minimal} if every orbit is dense, i.e. for every $x\in
X$,
$$\overline{\{f^n(x):n\in\ZZ\}}=X.$$
\end{definition}

\begin{definition}
For $n\in\ZZ^+$ and $\ep>0$, a subset $E\subset X$ 
is called an \emph{$(n,\ep)$-separated} if
for any distinct points $x,y$ in $E$, there is $k\in\{0,\cdots,n-1\}$ such
that
$$d(f^k(x),f^k(y))>\ep.$$
Denote by $s(n,\ep)$ the maximal cardinality of $(n,\ep)$-separated subsets
of $X$. Then the \emph{topological entropy} of $f$ is defined as
$$h(f):=\lim_{\ep\to0}\limsup_{n\to\infty}\frac{\ln s(n,\ep)}{n}.$$
\end{definition}

\begin{definition}\label{gapshadow}
 We call the finite sequence of ordered pairs
$$\sC=\{(x_j,m_j)\in X\times\ZZ^+: 
j=1,\cdots, k\}$$
an \emph{orbit sequence of rank $k$}. A \emph{gap} for an orbit sequence
of rank $k$ is a $(k-1)$-tuple
$$\sg=\{t_j\in\ZZ^+: j=1,\cdots, k-1\}.$$
For $\ep>0$, we say that $(\sC,\sg)$ can be \emph{$\ep$-shadowed} by $z\in
X$ if
for every $j=1,\cdots,k$,
$$(f^{s_{j}+l}(z), f^l(x_j))<\ep\text{ for every }l=0,1,\cdots, m_j-1,$$
where
$$s_1=0\text{ and }s_j=\sum_{i=1}^{j-1}(m_i+t_i-1)\text{ for }j=2,\cdots,k.$$
\end{definition}

\begin{definition}\label{spec}
$(X,f)$ is said to have  specification property if for every $\ep>0$,
there is $M(\ep)>0$ such that any $(\sC,\sg)$ with $\min\sg\ge M(\ep)$ can be $\ep$-shadowed.
\end{definition}

\begin{definition}
$(X,f)$ is said to have  periodic 
specification property if for every $\ep>0$,
there is $M(\ep)>0$ such that for any $t\ge M(\ep)$,
 any $(\sC,\sg)$ with $\min\sg\ge M(\ep)$ can be 
 $\ep$-shadowed by a periodic point of the period $s_k+t$.
\end{definition}

\begin{definition}\label{defgo}
$(X,f)$ is said to  have \emph{gluing orbit property}
 if for every $\ep>0$ there
is $M(\ep)>0$ such that for any orbit sequence 
$\sC$, there is a gap $\sg$ 
such that $\max\sg\le M(\ep)$ and $(\sC,\sg)$  can be $\ep$-shadowed.
\end{definition}

\begin{definition}\label{perglu}
$(X,f)$ is said to have \emph{periodic gluing orbit property} 
if for every $\ep>0$,
there is $M(\ep)>0$ such that 
for any orbit sequence $\sC$, there are $t\le M(\ep)$ and
a gap $\sg$ with $\max\sg\le M(\ep)$ such that $(\sC,\sg)$ can be $\ep$-shadowed
by a periodic point of the period $s_k+m_k+t$.
\end{definition}

The notion of specification property was first introduced by Bowen in \cite{Bowen}.
It has a number of variations and their names also vary in different literatures.
An overview of these specification like properties can be found in \cite{KLO}.
Gluing orbit property first appeared in \cite{ST}, where it is called transitive
specification. It is called weak specification in \cite{CT} in a slightly
generalized form. It is in \cite{BV} that the name gluing orbit is called
to indicate its dissimilarity with specification like properties.

Here we attempt to reformulate the definitions of specification and gluing
orbit properties to make our argument more
clear and more convenient. We follow the names called in 
\cite{BV}, \cite{KLO} and \cite{TWW}. Note that in our definitions of
periodic specification and periodic gluing orbit properties the
gap $\sg$ may be $\emptyset$. 

Definition \ref{gapshadow} naturally extends to infinite orbit sequences.
Definition \ref{spec} and \ref{defgo} are conventional definitions speaking
of finite
orbit sequences. However, they are equivalent to the definitions speaking
of infinite ones. This is clear for specification. For gluing orbit property
a little extra work should be done.
The flow version of the following lemma is contained in \cite{CLT}. A similar
technique is also part of the proof of Theorem \ref{posent}. 

\begin{lemma}\label{shadowinf}
$(X,f)$ has gluing orbit property if and only if for every $\ep>0$, there
is $L(\ep)>0$ such that for any infinite orbit
sequence $\sC=\{(x_j,m_j):j\in\ZZ\}$, there is $\sg$ with $\max\sg\le L(\ep)$,
$(\sC,\sg)$ can be $\ep$-shadowed. Moreover, we can take $L(\ep)\le M(\ep')$
for any $\ep'<\ep$.
\end{lemma}

\begin{proof} The if part is trivial.

Assume that $(X,f)$ has the gluing
orbit property as defined in Definition \ref{defgo}.
Let $\ep'<\ep$, $m=M(\ep')$ and 
$\sC=\{(x_j,m_j):j\in\ZZ^+\}$ be any infinite (forward) orbit
sequence. Proof for two-sided infinite sequences is analogous.
We denote for $k\ge 2$,
$$\sC_k=\{(x_j,m_j):j=1,\cdots, k\}.$$
For each $k\ge 2$, there is $\sg_k=\{t_1(k),\cdots, t_{k-1}(k)\}$ with
$\max\sg_k\le m$ and
$z_k\in X$ such that $(\sC_k,\sg_k)$ is $\ep'$-shadowed by $z_k$.

There is $t_1\in\{1,\cdots,m\}$ and a subsequence $\{z_{n(1,k)}\}$ of $\{z_k\}$
 such that
 $$t_1(n(1,k))=t_1\text{ for every }k.$$
There is $t_2\in\{1,\cdots,m\}$ and a subsequence $\{z_{n(2,k)}\}$ of $\{z_{n(1,k)}\}$
 such that
 $$t_2(n(2,k))=t_2\text{ for every }k.$$
Apply this procedure inductively, we obtain a sequence $\sg=\{t_j\}_{j=1}^\infty$
and subsequence $\{z_{n(j,k)}\}$
for each $j\in\ZZ^+$ such that
 $$t_j(n(j,k))=t_j\text{ for every }k.$$
Let $z$ be a subsequential limit of $\{z_{n(k,k)}\}$. Then
for every $j\in\ZZ^+$ and $l=0,1,\cdots, m_j-1$,
$$d(f^{s_j+l}(z),f^l(x_j))\le\limsup_{k\to\infty\ (k>j)} 
d(f^{s_j+l}(z_{n(k,k)}), f^l(x_{j}))\le\ep'<\ep,
$$
where
$$s_1=0\text{ and }s_j=\sum_{i=1}^{j-1}(m_i+t_i-1)\text{ for }j\ge2.$$
So $(\sC,\sg)$ is $\ep$-shadowed by $z$ and $\max\sg\le m$.
\end{proof}

Initial idea of the proof of Theorem \ref{posent} comes from the following
classical result. 

\begin{proposition}[cf. \cite{Bowen}]
Assume that $(X,f)$ has specification property.
Assume that $\ep>0$ and there is a subset $E$ of $X$ that is $(1,3\ep)$-separated
and $|E|=N\ge 2$. Then
$$h(f)\ge\frac{\ln N}{M(\ep)}.$$
\end{proposition}

\begin{proof}Let $m=M(\ep)$.
For $n\in\ZZ^+$ and each $\xi=\{x_1(\xi),\cdots, x_n(\xi)\}\in E^n$,
let
$$\sC_\xi=\{(x_j(\xi),1):j=1,\cdots,n\}$$
and $\sg_n=\{m,m,\cdots,m\}$.
There is $z_\xi\in X$ that $\ep$-shadows
$(\sC_\xi,\sg_n)$. If $\xi\ne\xi'$ then
there is $j\in\{1,\cdots,n\}$ such that
$$d(x_j(\xi),x_j(\xi'))>3\ep$$
and hence
$$d(f^{s_j}(z_\xi), f^{s_j}(z_{\xi'}))\ge d(x_j(\xi),x_j(\xi'))-d(f^{s_j}(z_\xi),x_j(\xi))-d(f^{s_j}(z_{\xi'}),x_j(\xi'))>\ep.$$
This implies that
$$A_n=\{z_\xi:\xi\in E^n\}$$
is an $(mn,\ep)$-separated set and hence $s(mn,\ep)\ge|A_n|=N^n$. This yields
that
$$h(f)\ge\limsup_{n\to\infty}\frac{\ln s(mn,\ep)}{mn}\ge\frac{\ln N}{m}.$$
\end{proof}

\begin{corollary}Assume that $(X,f)$ has specification property and $X$ is
not a singleton. Then the followings hold.
\begin{enumerate}
\item $h(f)>0$.
\item $\lim_{\ep\to 0}M(\ep)=\infty.$
\item Denote the lower box dimension of $X$ by
$$\dim X:=\liminf_{\ep\to 0}-\frac{\ln s(1,\ep)}{\ln\ep}.$$
Then
$$\liminf_{\ep\to 0}-\frac{M(\ep)}{\ln\ep}\ge\frac{\dim X}{h(f)}.$$
\end{enumerate}
\end{corollary}

\section{Positive Entropy}

Recall that a point $x\in X$ is called \emph{recurrent} if for every $\ep>0$
there is $n\in\ZZO,$ such that $d(f^n(x),x)<\ep$. A point is called 
\emph{non-recurrent} if it is not recurrent. Given a non-minimal system with
gluing orbit property, to show that it has positive topological
entropy, our idea  
is based on existence of two non-recurrent points such that the forward
 orbit
of one point stays away from the other point, and vice versa.

Note that 
without gluing orbit property,
a non-minimal system 
may have no non-recurrent points
and a system with non-recurrent points may have
zero topological entropy (cf. Example \ref{expallrec} and \ref{expole}).

\begin{lemma}\label{nonrec}
Assume that $(X, f)$ is not minimal and has  gluing orbit property. 
Then $f$
has a non-recurrent point.
\end{lemma}

\begin{proof}



As $f$ is not minimal, there is a point whose orbit is not dense.
We can find $x,y\in X$ and $\delta>0$
such that
$$d(f^n(x),y)\ge\delta\text{ for every }n\in\ZZ.$$
Let $0<\ep<\frac13\delta$. Assume that for every orbit sequence
$\sC$ there is a gap
$\sg$ with $\max\sg\le M(\ep)$ such that $(\sC,\sg)$ is $\ep$-shadowed.
Let $m=M(\ep)$. 
For each $n\in\ZZ^+$, consider
$$\sC_n=\{(f^{-(n-1)}(x),n),(y,1),(x,n)\}.$$
There is $\sg_n\in\{1,\cdots,m\}^2$ such that
$(\sC_n,\sg_n)$ is $\ep$-shadowed by $z_n'$.
There must be $\sg=(t_1,t_2)\in \{1,\cdots,m\}^2$ such that
$$\{n\in\ZZ^+:\sg_n=\sg\}\text{ is infinite}.$$
We can find a subsequence $\{z_{n_k}'\}$ with $\sg_{n_k}=\sg$ for every $k$.
Let $z_{n_k}=f^{n_k+t_1-1}(z_{n_k}')$ and
$z$ be a subsequential limit of $\{z_{n_k}\}$. Then
$$d(f^{-t_1-j}(z), f^{-j}(x))\le
\limsup_{n_k\to\infty}d(f^{-t_1-j}(z_{n_k}), f^{-j}(x))
\le\ep\text{ for every }j\ge 0$$
and
$$d(f^{t_2+j}(z), f^{j}(x))\le
\limsup_{n_k\to\infty}d(f^{t_2+j}(z_{n_k}), f^{j}(x))
\le\ep\text{ for every }j\ge 0.$$
This implies that
$$d(f^j(z),y)\ge
\delta-\ep> 2\ep\text{ for every $j\le-t_1$ or $j\ge t_2$}.$$
But $d(z,y)<\ep$. So
$$d(f^j(z),z)\ge d(f^j(z),y)-d(z,y)> \ep\text{ for every $j\le-t_1$ or $j\ge
t_2$}.$$
This also indicates that $z$ is not periodic. 
So we have
$$\min\{d(f^j(z),z): -t_1<j<t_2, j\ne 0\}=\ep'>0.$$
Then $z$ is a non-recurrent point as
$$d(f^j(z),z)\ge\min\{\ep,\ep'\}>0\text{ for every }j\in\ZZO.$$
\end{proof}


\begin{lemma}\label{stayaway}
Assume that $(X, f)$ is not minimal and has gluing orbit property.
Then
there are $x,y\in X$ and $\ep>0$ such that
\begin{align*}
d(f^n(x),x)&\ge\ep\text{ for any }n\in\ZZO,\\
d(f^n(x),y)&\ge\ep\text{ for any }n\ge 0,\\
 d(f^n(y),x)&\ge\ep\text{ for any }n\ge 0,\text{ and }\\
 d(f^n(y),y)&\ge\ep\text{ for any }n>0.
\end{align*}
\end{lemma}

\begin{proof}
By Lemma \ref{nonrec}, there is a non-recurrent point $x\in X$.
Assume that $$d(f^n(x),x)\ge\delta\text{ for every $n\in\ZZO$}.$$ 
Let $\ep_1=\frac13\delta$ and 
$m_1=M(\ep_1)$. For each $n$, there is $t_n\le m_1$ 
such that
$$(\{(x,1),(x,n)\},\{t_n\})$$ is $\ep_1$-shadowed by $y_n'$.
Let $y_n=f^{t_n}(y_n')$. 
Then
$$d(f^j(y_n), f^j(x))<\ep_1\text{ for }j=0,1,\cdots, n-1.$$
As $t_n\in\{1,\cdots,m_1\}$ for every $n$, there is $t$ such that
$$\{t_n: t_n=t\}\text{ is infinite}.$$
We can find a subsequence $\{y_{n_k}\}$ such that
$$t_{n_k}=t\text{ for every }k.$$
Let $y$ a a subsequential limit of $\{y_{n_k}\}$. Then
$$d(f^j(y), f^j(x))\le\limsup_{n_k\to\infty}d(f^{j}(y_{n_k}), f^{j}(x))
\le\ep_1\text{ for every }j\ge 0.$$
Moreover
$$d(f^{-t}(y),x)\le\limsup_{n_k\to\infty}d(y_{n_k}',x)\le\ep_1<d(f^{-t}(x),x),$$
which guarantees that $y\ne x$.
Let $$\ep:=d(x,y)=d(f^0(x),y)=d(f^0(y),x)\le\ep_1.$$ 
Then for every $n>0$,
\begin{align*}
d(f^n(x),y)&\ge  d(f^n(x),x)-d(x,y)\ge\delta-\ep\ge\ep,\\
d(f^n(y),x)&\ge  d(f^n(x),x)-d(f^n(x),f^n(y))\ge\delta-\ep_1\ge\ep,\\
d(f^n(y),y)&\ge  d(f^n(x),x)-d(f^n(y),f^n(x))-d(x,y)\ge\delta-\ep_1-\ep\ge\ep.
\end{align*}
\end{proof}

Now we complete the proof of Theorem \ref{posent}.
Let $x,y\in X$ and $\ep>0$ be as in Lemma \ref{stayaway}.
Let $\ep_2=\frac13\ep$ and $m=M(\ep_2)$. For each 
$\xi=\{x_k(\xi)\}_{k=1}^n\in\{x,y\}^n$, consider
$$\sC_\xi=\{(x_k,m):k=1,\cdots, n\}.$$
There is $$\sg_\xi=\{t_j(\xi):j=1,\cdots,n-1\}$$ 
with $\max\sg_\xi\le m$ such
that $(\sC,\sg)$ is $\ep_2$-shadowed by $z_\xi\in X$.
We claim that if $\xi\ne\xi'$ then there is $s<2mn$ such that
$$d(f^s(z_\xi),f^s(z_{\xi'}))>\ep_2.$$

If $\sg_\xi=\sg_{\xi'}$, then there is $k$ such that
$x_k(\xi)\ne x_k(\xi')$. For 
$$s:=\sum_{j=1}^{k-1}(m+t_j(\xi)-1)\le (2m-1)(k-1)<2mn$$
we have
\begin{align*}
d(f^s(z_\xi),f^s(z_{\xi'}))&
\ge 
d(x_k(\xi),x_k(\xi'))-d(f^s(z_\xi),x_k({\xi}))-d(f^s(z_{\xi'}),x_k(\xi'))
\\&>d(x,y)-2\ep_2\ge\ep_2.
\end{align*}
If $\sg_\xi\ne\sg_{\xi'}$, we may assume that there is $k$ such that
$$t_j(\xi)=t_j(\xi')\text{ for }j<k\text{ and }t_k(\xi)>t_k(\xi').$$
Let $l=t_k(\xi)-t_k(\xi')$. Then $1\le l\le m-1$. For
$$s:=\sum_{j=1}^{k-1}(m+t_j(\xi)-1)=\sum_{j=1}^{k-1}(m+t_j(\xi')-1)+l\le (2m-1)(k-1)<2mn$$
we have
\begin{align*}
d(f^s(z_\xi),f^s(z_{\xi'}))&
\ge 
d(x_k(\xi),f^l(x_k(\xi')))-d(f^s(z_\xi),x_k({\xi}))-d(f^s(z_{\xi'}), f^l(x_k(\xi')))
\\&>\min\{d(f^l(x),x), d(f^l(x),y),  d(f^l(y),x),  d(f^l(y),y)\}-2\ep_2
\\&\ge\ep_2.
\end{align*}

Above argument shows that 
$$E=\{z_\xi: \xi\in\{x,y\}^n\}$$
is a $(2mn,\ep_2)$-separated subset of $X$
that contains $2^n$ points.
Hence
$$h(f)\ge\limsup_{n\to\infty}\frac{\ln s(2mn,\ep_2)}{2mn}\ge\limsup_{n\to\infty}\frac{n\ln
2}{2mn}=\frac{\ln2}{2m}>0.$$

\section{Unique Ergodicity and Growth of Periodic Orbits}

\begin{theorem}\label{unierg}
Assume that $(X,f)$ is uniquely ergodic  and has gluing orbit property. Then
it is minimal.
\end{theorem}

\begin{proof}
Assume that $(X,f)$ is not minimal and it has gluing orbit property. There are $x,y\in X$ and $\delta>0$
such that
$$d(f^n(x),y)\ge\delta\text{ for every }n\in\ZZ.$$

Let $0<\ep'<\ep<\frac13\delta$ and $m=M(\ep')$.
Let $$\sC=\{(x_j,1): j\in\ZZ^+\text{ and }x_j=y\text{ for every }j
\}.$$ By Lemma \ref{shadowinf},
there is $y_0\in X$ that $\ep$-shadows 
$(\sC,\sg)$ for some $\sg$ with $\max\sg\le m$.

Take a continuous function $\varphi:X\to\RR$
such that
\begin{align*}
\varphi(z)=1&\text{ for every }z\in\overline{B(y,\ep)}\\
\varphi(z)=0&\text{ for every }z\notin B(y,2\ep)\\
0<\varphi(z)<1&\text{ otherwise. }
\end{align*}
We have
$$\limsup_{n\to\infty}\frac1n\sum_{k=0}^{n-1}\varphi(f^k(x))=0.$$
But
$$\liminf_{n\to\infty}\frac1n\sum_{k=0}^{n-1}\varphi(f^k(y_0))\ge\frac1m$$
as the orbit of $y_0$ enters $\overline{B(y,\ep)}$ at least once in every $m$ iterates.
This implies that $(X,f)$ is not uniquely ergodic.
\end{proof}

Denote by $P_n(f)$ the set of periodic points of $f$ with periods no more
than $n$, and $p_n(f)$ the cardinality of $P_n(f)$. Consider
$$p(f)=\limsup_{n\to\infty}\frac{\ln p_n(f)}{n}.$$
A flow version of the following theorem 
is contained in \cite{BTV}.

\begin{theorem}\label{entper}
If $(X,f)$ has  periodic gluing orbit property, then
$h(f)\le p(f)$.
\end{theorem}

\begin{proof} 
Assume that $h<h(f)$. There is $\ep>0$ and $N>0$ such that $s(n,\ep)>e^{nh}$
for every $n>N$.
Let $E$ be an $(n,\ep)$-separated set with $|E|>e^{nh}$.
Denote
$m=M(\dfrac\ep2)$.
For every $x\in E$, there is $t<m$ such that $\{(x,n),\emptyset\}$
is $\dfrac\ep2$-shadowed by a periodic point with period $n+t$.
Hence every $(n,\dfrac\ep2)$-ball around an element of $E$, which are disjoint
with each other, contains an element
of $P_{n+m}(f)$. This implies that $$p_{n+m}(f)\ge|E|>e^{nh}.$$
It follows that
$$p(f)=\limsup_{n\to\infty}\frac{\ln p_n(f)}{n}\ge h.$$
The result follows as this holds for any $h<h(f)$.
\end{proof}



\begin{corollary}
Assume that $(X,f)$ has  periodic gluing orbit property and $X$ does not
consist of a single periodic orbit. Then
$$0<h(f)\le p(f).$$
\end{corollary}


\section{Equicontinuous Systems}

Let $(X,f)$ be  an equicontinuous system . We shall show that
minimality implies gluing orbit property. It is clear that gluing orbit
property implies topological transitivity. For completeness, 
we present a proof that
topological transitivity implies minimality. As every equicontinuous system
has zero topological entropy, the fact that gluing orbit property implies
minimality is also a corollary of Theorem \ref{posent}.

We first prove a lemma that shows that the time needed for the pre-images of $\ep$-balls
to cover $X$ is uniform. We remark that this lemma does not require equicontinuity.

\begin{lemma}\label{lemmacover}
Assume that $(X,f)$ is minimal. Then for every $\ep>0$, there is $N\in\ZZ^+$
 such
that for every $x\in X$,
$$\bigcup_{n=0}^N f^{-n}(B(x,\ep))=X.$$
\end{lemma}

\begin{proof}
Let $\ep>0$ and $x\in X$. As $f$ is minimal, for every $y\in X$, there
is $n\in\ZZ$ such that $f^n(y)\in B(x,\ep)$. Equivalently, $y\in f^{-n}(B(x,\ep))$.
This implies that
$$X\subset\bigcup_{n=-\infty}^\infty f^{-n}(B(x,\ep)).$$
As $X$ is compact, there is $N_x=2N_x'$ such that
$$X\subset\bigcup_{n=-N_x'}^{N_x'} f^{-n}(B(x,\ep))$$
and hence
\begin{equation}\label{eqcover}
X=f^{-N_x'}(X)\subset\bigcup_{n=0}^{N_x} f^{-n}(B(x,\ep))
\end{equation}

For every $y\in X$, 
denote $$r(y)
:=\max\{d(f^n(y),x):{0\le n\le N_x}\text{ such
that }d(f^n(y),x)<\ep.\}$$
By \eqref{eqcover}, we have $r(y)<\ep$ for every $y\in X$. We claim that the function
$r:X\to\RR$ is upper semi-continuous.
 
Assume that $y\in X$ and $r(y)=d(f^{n_y}(y),x)<\ep$. Then for every $\ep'>0$, there is
$\delta>0$ such that for every $z\in B(y,\delta)$,
$$d(f^{n_y}(z),f^{n_y}(y))<\min\{\ep',\ep-r(y)\}.$$
Then
$$d(f^{n_y}(z),x)\le d(f^{n_y}(y),x)+d(f^{n_y}(z),f^{n_y}(y))<\ep.$$
This implies that
$$r(z)\ge d(f^{n_y}(z),x)\ge d(f^{n_y}(y),x)-d(f^{n_y}(z),f^{n_y}(y))>r(y)-\ep'.$$

As $r$ is upper semi-continuous and $X$ is compact, $r$ attains its maximum
$R_x<\ep$ on $X$. Let $\delta_x:=\dfrac{\ep-R_x}{2}>0$. Then for
every $x'\in B(x,\delta_x)$, we have
\begin{align*}
&\min\{d(f^n(y),x'):{0\le n\le N_x}\}
\\\le&\min\{d(f^n(y),x)+d(x,x'):{0\le n\le N_x}\}
\\=&\min\{d(f^n(y),x):{0\le n\le N_x}\}+d(x,x')
\\\le&r(y)+\delta_0
\le R_x+\delta_0
<\ep
\end{align*}
for every $y\in X$.
This implies that
$$X\subset\bigcup_{n=0}^{N_x} f^{-n}(B(x',\ep))\text{ for every }x'\in B(x,\delta_x).$$

Note that $\{B(x,\delta_x):x\in X\}$ is an open cover of $X$. It has a finite
subcover $\{B(x_j,\delta_{x_j}): j=1,\cdots,k\}$. Let $N=\max\{N_{x_j}:j=1,\cdots,k\}$.
Then for every $x\in X$, $x\in B(x_j,\delta_{x_j})$ for some $j$ and hence
$$X\subset\bigcup_{n=0}^{N_{x_j}} f^{-n}(B(x,\ep))\subset\bigcup_{n=0}^{N} f^{-n}(B(x,\ep)).$$

\end{proof}

The proof of Theorem \ref{equieq} is completed by 
Proposition \ref{minigo} and Proposition \ref{transmini}.

\begin{proposition}\label{minigo}
A minimal equicontinuous system has gluing orbit property. 
\end{proposition}

\begin{proof}
Let $\ep>0$. By equicontinuity,
there is $\delta>0$ such that
\begin{equation}\label{eqeqc}
d(f^n(x),f^n(y))<\ep\text{ whenever }d(x,y)<\delta.
\end{equation}
 By Lemma
\ref{lemmacover}, there is $M$ such that
$$\bigcup_{n=0}^M f^{-n}(B(x,\delta))=X\text{ for every }x\in X.$$

Let $\sC=\{(x_j,m_j):j=1,\cdots,k\}$ be any orbit chain. We claim that there
is a gap $\sg$ with $\max\sg\le M+1$ such that $(\sC,\sg)$ can be $\ep$-shadowed
by $x_1$.

For each $j=1,\cdots, k-1$, we have
$$\bigcup_{n=1}^{M+1} f^{-n}(B(x_{j+1},\delta))=f^{-1}(\bigcup_{n=0}^M f^{-n}(B(x_{j+1},\delta)))=X\ni f^{s_j+m_j-1}(x_1),$$
where
$$s_1=0\text{ and }s_j=\sum_{i=1}^{j-1}(m_i+t_i-1)\text{ for }j=2,\cdots,k.$$
There is $t_j\in\ZZ^+$ such that 
$$t_j\le M+1\text{ and }f^{t_j}(f^{s_j+m_j-1}(x_1))\in
B(x_{j+1},\delta).$$
By \eqref{eqeqc}, this implies that
$$(f^{s_{j+1}+l}(x_1), f^l(x_{j+1}))<\ep\text{ for every }l=0,1,\cdots, m_j-1.$$
Hence $(\sC,\sg)$ is $\ep$-shadowed by $x_1$ for $\sg=\{t_j:j=1,\cdots, k-1\}$.
\end{proof}


\begin{proposition}\label{transmini}
A topological transitive equicontinuous system is minimal.
\end{proposition}

\begin{proof}
Let $x,y\in X$ and $\ep>0$. As $f$ is equicontinuous, there is $\delta>0$
such that
$$d(f^n(z),f^n(x))<\frac\ep2\text{ for every }z\in B(x,\delta).$$
As $f$ is topologically transitive, there is $n\ge 0$ such that
$$B(x,\delta)\cap f^{-n}(B(y,\frac\ep2))\ne\emptyset.$$
Take 
$$z_0\in B(x,\delta)\cap f^{-n}(B(y,\frac\ep2)).$$
Then
$$d(f^n(x),y)\le d(f^n(x),f^n(z_0))+d(f^n(z_0),y)<\frac\ep2+\frac\ep2=\ep.$$
This implies that the orbit of every $x\in X$ is dense, i.e. $f$ is minimal.
\end{proof}

\section{Examples}

\begin{example}
In \cite{BV}, it is shown that a topologically transitive subshift of finite
type has gluing orbit property. Note that such a system has periodic points.
As a corollary of Theorem \ref{posent}, it has positive topological entropy
if it does not consist of a single periodic orbit.
\end{example}

\begin{example}
As a corollary of Theorem \ref{equieq}, adding machines have gluing orbit
property.
\end{example}

\begin{example}\label{expallrec}
A rational rotation is not minimal but every point is recurrent 
(cf. Lemma \ref{nonrec}). 
\end{example}

\begin{example}\label{expole}
Consider the map
$$f(x)=x^2 \mod 1$$
on the unit circle $[0,1]/\sim$.
It is not minimal and has non-recurrent points and zero topological entropy
(cf. Theorem \ref{posent} and Lemma \ref{stayaway}). It is also uniquely
ergodic (cf. Theorem \ref{unierg}). It fails to satisfy the theorems as it
does not have gluing orbit property.
\end{example}

\begin{example}\label{exintmap}
According to \cite{MS}, there are $C^\infty$ interval maps with
periodic points of period $2^n$ for any $n\in\ZZ^+$ 
and zero topological entropy that are chaotic in the sense
of Li-York. Theorem \ref{posent} implies that all such maps can not have gluing
orbit property.
\end{example}

\begin{example}\label{exaper}
The subshift on the closure of the orbit of an almost periodic point, as
constructed in \cite[12.28]{GH} and \cite{HK},
does not have gluing orbit property. The gap needed before shadowing an
orbit segment of length $L$ may be no less than $L$ and hence neither uniform
nor tempered.
Such a system is minimal. It can have zero
topological entropy and can also have positive topological entropy. So minimality
itself does not imply gluing orbit property, no matter how much is the topological
entropy (cf. Theorem \ref{equieq}).
\end{example}

\section{The Semiflow Case}

In this section we give a proof of Theorem \ref{posent} in the semiflow
case. Throughout this section, $f^t$ is assumed to be a semiflow on $X$ that
is not minimal and has gluing orbit property. We first state the definition
of gluing orbit property in this case and note the
difference. Idea of the proof is similar to the homeomorphism case. There
are two major technical differences. Non-recurrence is established after
a time period and the orbit sequences for finding separated sets are more
carefully designed.

\begin{definition}
A semiflow $(X,f)$ is said to  have \emph{gluing orbit property}
 if for every $\ep>0$ there
is $M(\ep)>0$ such that for any \emph{orbit sequence}
$$\sC=\{(x_j,m_j)\in X\times[0,\infty):j=1,\cdots,k\},$$ 
there is a \emph{gap}
$$\sg=\{t_j\in[0,\infty): j=1,\cdots,k-1\}$$ 
such that $\max\sg\le M(\ep)$ and $(\sC,\sg)$  can be $\ep$-shadowed in the
following sense:
for every $j=1,\cdots,k$,
$$(f^{s_{j}+t}(z), f^t(x_j))<\ep\text{ for every }t\in[0,m_j],$$
where
$$s_1=0\text{ and }s_j=\sum_{i=1}^{j-1}(m_i+t_i)\text{ for }j=2,\cdots,k.$$
\end{definition}


\begin{lemma}\label{nonrecsf}
There is $x_0\in X$, $\ep>0$ and $\tau>0$ such that
$$d(f^t(x_0),x_0)>\ep\text{ for any }t\ge\tau.$$
\end{lemma}

\begin{proof}
As $f$ is not minimal, there is a point whose orbit is not dense.
We can find $x,y\in X$ and $\delta>0$
such that
$$d(f^t(x),y)\ge\delta\text{ for every }t\ge0.$$
Let $0<\ep<\frac13\delta$ and $m=M(\ep)$. 
For each $n\in\ZZ^+$, consider
$$\sC_n=\{(y,0),(x,n)\}.$$
There is $\tau_n\in[0,m]$ such that
$(\sC_n,\{\tau_n\})$ is $\ep$-shadowed by $z_n$.
There must be a subsequence $\{\tau_{n_k}\}$ that converges to
 $\tau\in[0,m]$. 
Let 
$x_0$ be a subsequential limit of $\{z_{n_k}\}$. Then
$$d(f^{\tau+t}(x_0), f^{t}(x))\le\limsup_{n_k\to\infty}
d(f^{\tau_{n_k}+t}(z_{n_k}), f^{t}(x))\le\ep\text{ for every }t\ge 0.$$
and
$$d(f^{\tau+t}(x_0),y)\ge d(f^t(x),y)- d(f^{\tau+t}(x_0), f^{t}(x))
>2\ep\text{ for every $t\ge0$}.$$
Note that $d(x_0,y)<\ep$. So
$$d(f^{\tau+t}(x_0),x_0)\ge d(f^{\tau+t}(x_0),y)-d(x_0,y)>
\ep\text{ for every $t\ge0$}.$$
\end{proof}


\begin{lemma}\label{stayawaysf}
There are $x,y\in X$, $\ep>0$ and $T>0$ such that
\begin{align*}
d(f^t(x),x)&\ge\ep\text{ for any }t\ge T,\\
d(f^t(y),x)&\ge\ep\text{ for any }t\ge T,\\
 d(f^t(y),y)&\ge\ep\text{ for any }t\ge T,\text{ and }\\
 d(f^t(x),y)&\ge\ep\text{ for any }t\ge0.
\end{align*}
\end{lemma}

\begin{proof}
By Lemma \ref{nonrecsf}, there is $x\in X$, $\delta>0$ and $t_0>0$ such
that
$$d(f^t(x),x)\ge\delta\text{ for every $t\ge t_0$}.$$ 
Let $\ep_1=\frac13\delta$ and 
$m_1=M(\ep_1)$. For each $n$, there is $\tau_n\in[0, m_1]$ 
such that
$$(\{(x,t_0),(x,n)\},\{\tau_n\})$$ is $\ep_1$-shadowed by $y_n$.
There is a subsequence $\{\tau_{n_k}\}$ that converges to $\tau\in[0,m_1]$.
Let $y$ be a subsequential limit of $\{y_{n_k}\}$. Then
\begin{equation}\label{followx}
d(f^{t_0+\tau+t}(y), f^{t}(x))\le\ep_1\text{ for every }t\ge 0.
\end{equation}
This yields that for every $t\ge 0$, we have
\begin{align}
d(f^{2t_0+\tau+t}(y),x)&
\ge d(f^{t_0+t}(x),x)-d(f^{2t_0+\tau+t}(y),f^{t_0+t}(x))
\ge\delta-\ep_1=2\ep_1,\notag\\
d(f^{2t_0+\tau+t}(y),y)&
\ge d(f^{2t_0+\tau+t}(y),x)-d(x,y)
\ge 2\ep_1-\ep_1=\ep_1,\text{ and}\notag\\
d(f^{t_0+\tau+t}(x),y)&\ge d(f^{t_0+\tau+t}(x),x)-d(x,y)\ge
\delta-\ep_1= 2\ep_1.\label{awayy}
\end{align}
Equation \eqref{followx} also guarantees that $f^t(x)\ne y$ for any $t\ge 0$, as
$$d(f^{t_0+\tau}(y), x)\le\ep_1<\delta\le d(f^{t_0+\tau+t}(x),x).$$
Let 
$$\ep:=\min\{d(f^t(x),y):0\le t\le t_0+\tau\}.$$
Then $\ep\in(0,\ep_1]$. Together with \eqref{awayy} we have
$$d(f^t(x),y)\ge\ep\text{ for every }t\ge 0.$$
The lemma holds for $x,y,\ep$ and $T=2t_0+\tau$.
\end{proof}

\begin{proposition}
$(X,f)$ has positive topological entropy.
\end{proposition}

\begin{proof}
Let $x,y\in X$, $\ep>0$ and $T>0$ be as in Lemma \ref{stayawaysf}.
Let $0<\ep_2<\frac13\ep$ and $m=M(\ep_2)$.
Let
$$Q_1=\{(y,2T+3m)\}\text{ and }Q_2=\{(x,T+m),(x,T+m)\}.$$

For each 
$\xi=\{\omega_k(\xi)\}_{k=1}^n\in\{1,2\}^n$, consider
$$\sC_\xi=\{Q_{\omega_k(\xi)}:k=1,\cdots, n\}=\{(x_j(\xi),m_j(\xi)):j=1,\cdots,n(\xi)\},$$
where
$$n(\xi)=\sum_{k=1}^n\omega_k(\xi).$$
There is $\sg_\xi=\{t_j(\xi):j=1,\cdots,n(\xi)-1\}$ 
with $\max\sg_\xi\le m$ such
that $(\sC,\sg)$ is $\ep_2$-shadowed by $z_\xi\in X$.
For each $\xi$, denote
$$s_1(\xi)=0\text{ and }s_j(\xi)=\sum_{i=1}^{j-1}(m_i(\xi)+t_i(\xi))\text{
for }j=2,\cdots,n(\xi).$$
Then
$$s_{n(\xi)}(\xi)\le(2T+4m)n\text{ for every }\xi\in\{1,2\}^n$$

We claim that if $\xi\ne\xi'$ then there is $s\le(2T+4m)n$ such that
$$d(f^s(z_\xi),f^s(z_{\xi'}))>\ep_2.$$

Assume that $x_j(\xi)=x_j(\xi')$ for $j=1,\cdots,l-1$,
$x_l(\xi)=y$ and  $x_l(\xi')=x$. 
Our discussion can be split into the following
cases.
\begin{enumerate}[{Case} 1.]
\item $l=1$. Then
$$d(z_\xi,z_{\xi'})\ge d(x,y)-d(z_\xi,x)-d(z_{\xi'},y)>\ep_2.$$
\item $l\ge2$ and there is $k<l$ with $|s_{k}(\xi)-s_{k}(\xi')|\ge T$.
Let $k$ be the smallest index satisfying the inequality. Then
$$|s_{k-1}(\xi)-s_{k-1}(\xi')|< T.$$
Then
\begin{align*}
r:={}&|s_{k}(\xi)-s_{k}(\xi')|\\
={}&|(s_{k-1}(\xi)+m_{k-1}(\xi)+t_{k-1}(\xi))-(s_{k-1}(\xi')+m_{k-1}(\xi')+t_{k-1}(\xi'))|\\
={}&|(s_{k-1}(\xi)+t_{k-1}(\xi))-(s_{k-1}(\xi')+t_{k-1}(\xi'))|\\
\le{}&|(s_{k-1}(\xi)-s_{k-1}(\xi')|+|t_{k-1}(\xi)-t_{k-1}(\xi'))|\\
<{} &T+m.
\end{align*}
Assume that $s_k(\xi)<s_k(\xi')$. Then
\begin{align*}
&d(f^{s_k(\xi')}(z_\xi), f^{s_k(\xi')}(z_{\xi'}))\\
\ge{} &d(f^r(x_k(\xi)),x_k(\xi'))-
d(f^{s_k(\xi)+r}(z_\xi),f^r(x_k(\xi)))-
d(f^{s_k(\xi')}(z_{\xi'}),x_k(\xi'))\\
>{}&\ep_2.
\end{align*}
\item $l\ge2$ and $|s_{l-1}(\xi)-s_{l-1}(\xi')|< T$.
A similar argument shows that
$$r:=|s_l(\xi)-s_l(\xi')|<T+m.$$
If $s_l(\xi)\ge s_l(\xi')$, then
\begin{align*}
&d(f^{s_l(\xi)}(z_\xi), f^{s_l(\xi)}(z_{\xi'}))\\
\ge{} &d(y,f^r(x))-
d(f^{s_l(\xi)}(z_\xi),y)-
d(f^{s_l(\xi')+r}(z_{\xi'}),f^r(x))\\
>{}&\ep_2.
\end{align*}
If $s_l(\xi)<s_l(\xi')$, then
$$r_1:=s_{l+1}(\xi')-s_l(\xi)=r+(T+m)+t_l(\xi')\in(T+m,2T+3m)$$
Note that $x_{l+1}(\xi')=x$. We have
\begin{align*}
&d(f^{s_{l+1}(\xi')}(z_\xi), f^{s_{l+1}(\xi')}(z_{\xi'}))\\
\ge{} &d(f^{r_1}(y),x)-
d(f^{s_{l}(\xi)+r_1}(z_\xi),f^{r_1}(y))-
d(f^{s_{l+1}(\xi')}(z_{\xi'}),x)\\
>{}&\ep_2.
\end{align*}
\end{enumerate}
Above argument shows that 
$$E=\{z_\xi: \xi\in\{1,2\}^n\}$$
is a $((2T+4m)n,\ep_2)$-separated subset of $X$
that contains $2^n$ points.
Hence
$$h(f)\ge\limsup_{n\to\infty}\frac{\ln s((2T+4m)n,\ep_2)}{(2T+4m)n}\ge
\limsup_{n\to\infty}\frac{n\ln
2}{(2T+4m)n}=\frac{\ln2}{2T+4m}>0.$$

\end{proof}

\section*{Acknowledgments}
The author is supported by NSFC No. 11571387. 
 


\end{document}